\documentclass[10pt, reqno]{amsart}

\usepackage{hyperref, cite, xcolor}
\usepackage{amsmath}
\usepackage{amsthm}
\usepackage{amssymb}
\usepackage{amsfonts}
\usepackage{latexsym}
\usepackage{setspace}
\usepackage{enumitem}
\usepackage{cancel}
\setlist[itemize]{leftmargin=*}
\setlist[enumerate]{leftmargin=*}

\usepackage[font=footnotesize,labelfont=bf]{caption}

\newcommand{\sgn}{\operatorname{sgn}}

\renewcommand{\Pi}{\mathbb{P}}

\newtheorem{theorem}{Theorem}

\newtheorem{lemma}[theorem]{Lemma}

\theoremstyle{definition}

\numberwithin{equation}{section}

\renewcommand{\phi}{\varphi}

\newcommand{\q}{\mathfrak{q}}

\renewcommand{\pmod}[1]{\,\,(\operatorname{mod} #1)}

\renewcommand{\geq}{\geqslant}
\renewcommand{\leq}{\leqslant}

\let\oldenumerate=\enumerate
	\def\enumerate{
	\oldenumerate
	\setlength{\itemsep}{5pt}
	}
\let\olditemize=\itemize
	\def\itemize{
	\olditemize
	\setlength{\itemsep}{5pt}
	}

\allowdisplaybreaks

\title[Primitive root biases for prime pairs I]{Primitive root biases for prime pairs I:\\  existence and non-totality of biases}
\author[S.R.~Garcia]{Stephan Ramon Garcia}
\address{Department of Mathematics\\Pomona College\\610 N. College Ave., Claremont, CA 91711} 
\email{stephan.garcia@pomona.edu}
\urladdr{http://pages.pomona.edu/~sg064747}
\thanks{SRG supported by NSF grant DMS-1265973 and the Budapest Semesters in Mathematics (BSM) Director's Mathematician in Residence (DMiR) program.
SRG and TS supported by a David L. Hirsch III and Susan H. Hirsch Research Initiation Grant.
FL was supported in part by grants CPRR160325161141 and an A-rated researcher award both from the NRF of South Africa and by grant no. 17-02804S of the Czech Granting Agency. }

\author[F.~Luca]{Florian Luca}
\address{School of Mathematics\\University of the Witwatersrand\\Private Bag 3, Wits 2050, Johannesburg, South Africa\\
Max Planck Institute for Mathematics, Vivatgasse 7, 53111 Bonn, Germany\\
Department of Mathematics, Faculty of Sciences, University of Ostrava, 30 dubna 22, 701 03
Ostrava 1, Czech Republic}
\email{Florian.Luca@wits.ac.za}

\author[T.~Schaaff]{Timothy Schaaff}
\email{timothyschaaff@gmail.com}


\begin{document}

\begin{abstract}
We study the difference between the number of primitive roots modulo $p$ and modulo $p+k$ for prime pairs $p,p+k$.  Assuming the Bateman--Horn conjecture, we prove the existence of strong sign biases for such pairs.  More importantly, we prove that for a small positive proportion of prime pairs $p,p+k$, the dominant inequality is reversed.
\end{abstract}

\maketitle

\section{Introduction}
Let $k$ be a positive even integer and
suppose that $p$ and $p+k$ are prime.  
Then the difference between the number of primitive roots modulo $p$ and modulo $p+k$ is
\begin{equation*}
	T(p):=\phi(p-1) - \phi(p+k-1).
\end{equation*}
If $T(p) > 0$, then $p$ has more primitive roots than $p+k$ does;
if $T(p) < 0$, then $p$ has fewer primitive roots than $p+k$ does.  
We are interested here in the sign of $T(p)$ as $p$ ranges over
the set of all primes $p$ for which $p+k$ is also prime.

To streamline our presentation, we let $\Pi_k$ denote the set of primes $p$
for which $p+k$ is prime.
For example, $\Pi_2$ is the set of twin primes, $\Pi_4$ is the set of
cousin primes, and $\Pi_6$ is the set of ``sexy primes.''  
We denote by $\Pi_k(x)$ the set of elements in $\Pi_k$ that are at most $x$.  
The number of elements in $\Pi_k(x)$ is denoted by $\pi_k(x)$; this is the counting function
of $\Pi_k$.
That is, $\pi_k(x)$ is the number of primes $p \leq x$ such that $p+k$ is prime. 
In what follows, the letters $p,q,r,s$ are reserved for primes.  

It has long been conjectured that each $\Pi_k$ is infinite
(this appears to date back at least to de Polignac).  For example, the 
\emph{twin-prime conjecture} asserts that $\Pi_2$ is infinite.  There have been
tantalizing steps toward this conjecture in recent years \cite{Maynard, Polymath, ZhangGaps}.
A more refined version of the twin-prime conjecture is the first \emph{Hardy--Littlewood conjecture},
which asserts that $\pi_2$ is asymptotic to a certain constant times $x/(\log x)^2$.
The far-reaching \emph{Bateman--Horn conjecture} (Section \ref{Section:BH}) implies that each
$\Pi_k$ is infinite and provides asymptotics for $\pi_k$ on the order of $x/(\log x)^2$.
The first Hardy--Littlewood conjecture and the twin-prime conjecture both follow from the
Bateman--Horn conjecture.

Our work is inspired by \cite{GKL}, in which a peculiar primitive root bias was discovered 
in the twin prime case $k=2$.
Assuming the Bateman--Horn conjecture, it was proved that at least $65.13\%$ 
of twin prime pairs $p,p+2$ satisfy $T(p) > 0$ and that at least $0.47\%$ satisfy $T(p) < 0$
(numerical evidence suggests that the bias is approximately $98\%$ to $2\%$).
This is interesting for two reasons.  First, a pronounced bias in favor of $T(p) > 0$
exists for twin primes (although relatively easy to motivate from a heuristic standpoint,
the proof is long and involved).
Second, the bias is not total:  the inequality is reversed for a
small positive proportion of the twin primes.

In this paper, we extend the results of \cite{GKL} to prime pairs $p,p+k$.
As before, we assume the Bateman--Horn conjecture.
Although there are some similarities,
many significant complications arise when passing from the case $k=2$ to $k \geq 4$.
\begin{enumerate}
\item The direction and magnitude of the bias in $T(p)$ now depend heavily on 
the value of $k \pmod{3}$ and the smallest primes that do not divide $k$.
If $k \equiv -1 \pmod{3}$, then an overwhelming majority of primes $p \in \Pi_k$
satisfy $T(p) > 0$.  If $k\equiv 1 \pmod{3}$, then the bias is strongly toward $T(p) < 0$.
If $k \equiv 0 \pmod{3}$, then the extreme bias disappears and either sign
can be favored.

\item An elementary lemma in the twin-prime case \cite[Lem.~2]{GKL} that relates the sign of 
$T(p)$ to the sign of a more tractable function fails 
for $k \geq 4$ and must be replaced by a much more difficult asymptotic version 
(Theorem \ref{Theorem:Comparison}).

\item The ``influence'' of the small primes $5$, $7$, and $11$
was sufficient to establish that a positive proportion of twin prime pairs $p,p+2$
satisfy $T(p) < 0$ \cite{GKL}.  This straightforward analysis 
is no longer possible for $k \geq 4$ and we must
introduce several parameters in order to compensate.

\item The tolerances are spectacularly small for certain $k$.
A notable example is $k=14$.
Among the first $20$ million primes there are 
$1{,}703{,}216$ pairs of primes of the form $p,p+14$; see Table \ref{Table:PrimePairList}.  Only three pairs
satisfy $T(p) \leq 0$, a proportion of $1.76 \times 10^{-6}$.
These sorts of numbers give us little room to maneuver.  
\end{enumerate}

A more extreme example is $k=70$.
Among the first $20$ million primes, every prime pair $p,p+70$ satisfies
$T(p) < 0$.  Nevertheless, our approach proves that a tiny positive proportion
(at least $1.81 \times 10^{-20}$) of the primes in $\Pi_{70}$ satisfy $T(p) > 0$.
Even in such lopsided cases, we are able to prove that the biases are not total: the
dominant inequality is reversed for a positive proportion of the primes considered.

\begin{table}\footnotesize
\begin{equation*}
\begin{array}{|c|ccl||c|ccl|}
\hline
k & \# T(p)<0 & \pi_k(x) & \text{Proportion} &
k & \# T(p)<0 & \pi_k(x) & \text{Proportion} \\
\hline
 2 & 28490 & 1418478 & 0.0200849 & 62 & 1980 & 1468111 & 0.00134867 \\
 4 & 1390701 & 1419044 & 0.980027 & 64 & 1416847 & 1418937 & 0.998527 \\
 6 & 1687207 & 2836640 & 0.594791 & 66 & 2187908 & 3153911 & 0.693713 \\
 \hline
 8 & 28771 & 1417738 & 0.0202936 & 68 & 25409 & 1512639 & 0.0167978 \\
 10 & 1891800 & 1891902 & 0.999946 & 70 & 2270424 & 2270424 & 1 \\
 12 & 1441259 & 2837946 & 0.507853 & 72 & 1431789 & 2837200 & 0.504649 \\
  \hline
 14 & 3 & 1703216 & 1.76 \times 10^{-6} & 74 & 64 & 1459313 & 4.39 \times 10^{-5} \\
 16 & 1420209 & 1420273 & 0.999955 & 76 & 1502310 & 1502338 & 0.999981 \\
 18 & 1433488 & 2837906 & 0.505122 & 78 & 1745211 & 3096187 & 0.563665 \\
  \hline
 20 & 4 & 1891034 & 2.12 \times 10^{-6} & 80 & 113 & 1892585 & 5.95 \times 10^{-5} \\
 22 & 1576076 & 1576379 & 0.999808 & 82 & 1426536 & 1455721 & 0.979952 \\
 24 & 1015032 & 2838360 & 0.357612 & 84 & 1145652 & 3404217 & 0.336539 \\
  \hline
 26 & 26521 & 1546675 & 0.0171471 & 86 & 28787 & 1454174 & 0.0197961 \\
 28 & 1699783 & 1702838 & 0.998206 & 88 & 1553144 & 1576531 & 0.985166 \\
 30 & 1930480 & 3784105 & 0.510155 & 90 & 1489160 & 3785003 & 0.393437 \\
  \hline
 32 & 20553 & 1418579 & 0.0144884 & 92 & 29413 & 1486659 & 0.0197846 \\
 34 & 1495332 & 1513933 & 0.987713 & 94 & 1421558 & 1450180 & 0.980263 \\
 36 & 2097416 & 2838465 & 0.738926 & 96 & 1915769 & 2839516 & 0.674682 \\
  \hline
 38 & 21739 & 1502517 & 0.0144684 & 98 & 377 & 1702580 & 0.000221429 \\
 40 & 1891651 & 1891659 & 0.999996 & 100 & 1891334 & 1891337 & 0.999998 \\
 42 & 1727098 & 3405081 & 0.507212 & 102 & 1531067 & 3027395 & 0.505737 \\
  \hline
 44 & 6 & 1576157 & 3.81 \times 10^{-6} & 104 & 9 & 1549054 & 5.81 \times 10^{-6} \\
 46 & 1486910 & 1486946 & 0.999976 & 106 & 1447486 & 1447486 & 1 \\
 48 & 1318068 & 2838746 & 0.464313 & 108 & 1434316 & 2838777 & 0.505258 \\
  \hline
 50 & 48 & 1891847 & 2.54 \times 10^{-5} & 110 & 16 & 2101919 & 7.61 \times 10^{-6} \\
 52 & 1525943 & 1548356 & 0.985525 & 112 & 1699877 & 1702796 & 0.998286 \\
 54 & 933772 & 2839928 & 0.328801 & 114 & 1051285 & 3004570 & 0.349895 \\
  \hline
 56 & 2272 & 1701628 & 0.00133519 & 116 & 22762 & 1471017 & 0.0154736 \\
 58 & 1447184 & 1472758 & 0.982635 & 118 & 1418455 & 1442208 & 0.98353 \\
 60 & 1939665 & 3783957 & 0.512602 & 120 & 2269102 & 3784749 & 0.599538 \\
 \hline
\end{array}
\end{equation*}
\caption{The proportion of prime pairs $p,p+k$ among the first $20$ million primes
for which $p$ has fewer primitive roots than $p+k$ does.
Extreme biases occur for $k \not \equiv 0 \pmod{3}$ 
(see Theorem \ref{thm:kpm1mod3}); the situation is more balanced
if $k \equiv 0 \pmod{3}$ (see Theorem \ref{thm:k0mod3}).}
\label{Table:PrimePairList}
\end{table}

This paper is organized as follows.
Section \ref{Section:BH} introduces the Bateman--Horn conjecture and a closely-related
unconditional result that is necessary for our work.
Section \ref{Section:Comparison} concerns a ``totient comparison theorem'' 
(Theorem \ref{Theorem:Comparison}) that permits us
to consider a more convenient function $S(p)$ in place of $T(p)$.
The short Section \ref{Section:Heuristic} contains an heuristic argument that
explains the dependence of our results upon the value of $k \pmod{3}$.
For $k \not \equiv 0 \pmod{3}$, the heuristic argument is turned into a rigorous, quantitative theorem in
Section \ref{Section:k1}, which contains our main result (Theorem \ref{thm:kpm1mod3}).
Although it is too technical to state here, Theorem \ref{thm:kpm1mod3} proves the following.
\begin{enumerate}
\item For $k \not \equiv 0 \pmod{3}$,
strong primitive root biases exist for prime pairs $p,p+k$. 

\item The biases are not total:  the dominant inequality is reversed for a 
positive proportion of prime pairs $p,p+k$.
\end{enumerate}
We conclude in Section \ref{Section:k0} with an analogous theorem 
(Theorem \ref{thm:k0mod3}) for $k \equiv 0 \pmod{3}$.
In this case, we prove that substantial positive proportions of $p \in \Pi_k$
satisfy $T(p) > 0$ and $T(p) < 0$, respectively.  Thus, the extreme biases observed
in the $k \not \equiv 0 \pmod{3}$ setting disappear.

\section{The Bateman--Horn conjecture and Brun's sieve}\label{Section:BH}

Let $f_1,f_2,\ldots,f_m$ be a collection of distinct irreducible polynomials with integer coefficients and
positive leading coefficients.  An integer $n$ is \emph{prime generating} for 
this collection if each $f_1(n), f_2(n),\ldots, f_m(n)$ is prime. 
Let $P(x)$ denote
the number of prime-generating integers at most $x$ and
suppose that $f = f_1f_2\cdots f_m$ does not vanish identically modulo any prime.
The \emph{Bateman--Horn conjecture} asserts that
\begin{equation*}
P(x)\, \sim\,  \frac{C}{D} \int _{2}^{x}\frac{dt}{(\log t)^m},
\end{equation*}
in which 
\begin{equation*}
D = \prod_{i=1}^m \deg f_i
\quad \text{and} \quad
C=\prod_{p} \frac{1-N_f(p)/p}{(1-1/p)^{m}},
\end{equation*}
in which $N_f(p)$ is the number of solutions to $f(n) \equiv 0 \pmod{p}$ \cite{Bateman}. 
For simplicity, we prefer the asymptotically equivalent expression
\begin{equation*}
	\frac{Cx}{D(\log x)^m}.
\end{equation*}

For a fixed $k$, let
\begin{equation}\label{eq:f}
	f(t) \,=\, t(t+k),
\end{equation}
so that
\begin{equation}\label{eq:BHNfp}
	N_f(p) \,=\, \begin{cases}
		1 & \text{if } p |  k, \\
		2 & \text{if }p \nmid k.
	\end{cases}
\end{equation}
The Bateman--Horn conjecture predicts that 
\begin{equation*}
	\pi_k(x) 
	\, \sim\, \prod_{p |  k} \frac{p(p-1)}{(p-1)^2} 
	\prod_{p \nmid k} \frac{p(p-2)}{(p-1)^2} \frac{x}{(\log x)^2} 
	\,=\, \frac{C_k x}{(\log x)^2},
\end{equation*}
in which
\begin{equation*}
	C_k = \prod_{p |  k} \frac{p(p-1)}{(p-1)^2} \prod_{p \nmid k} \frac{p(p-2)}{(p-1)^2}
\end{equation*}
depends only on upon the primes that divide $k$; see Table \ref{Table:CkList}.
For example, $C_k \approx 1.32032$ whenever $k$ is a power of $2$.  In particular, 
$C_k/2 \approx 0.660162$ is the \emph{twin-primes constant}. 

\begin{table}\small
\begin{equation*}
\begin{array}{|c|l||c|l||c|l||c|l|}
\hline
k & \multicolumn{1}{c||}{C_k} & k & \multicolumn{1}{c||}{C_k} & k & \multicolumn{1}{c||}{C_k} & k & \multicolumn{1}{c|}{C_k} \\
\hline\hline
 2 & 1.32032 & 32 & 1.32032 & 62 & 1.36585 & 92 & 1.3832 \\
 4 & 1.32032 & 34 & 1.40835 & 64 & 1.32032 & 94 & 1.34966 \\
 6 & 2.64065 & 36 & 2.64065 & 66 & 2.93405 & 96 & 2.64065 \\
 8 & 1.32032 & 38 & 1.39799 & 68 & 1.40835 & 98 & 1.58439 \\
 10 & 1.76043 & 40 & 1.76043 & 70 & 2.11252 & 100 & 1.76043 \\
 12 & 2.64065 & 42 & 3.16878 & 72 & 2.64065 & 102 & 2.81669 \\
 14 & 1.58439 & 44 & 1.46703 & 74 & 1.35805 & 104 & 1.44035 \\
 16 & 1.32032 & 46 & 1.3832 & 76 & 1.39799 & 106 & 1.34621 \\
 18 & 2.64065 & 48 & 2.64065 & 78 & 2.88071 & 108 & 2.64065 \\
 20 & 1.76043 & 50 & 1.76043 & 80 & 1.76043 & 110 & 1.95604 \\
 22 & 1.46703 & 52 & 1.44035 & 82 & 1.35418 & 112 & 1.58439 \\
 24 & 2.64065 & 54 & 2.64065 & 84 & 3.16878 & 114 & 2.79598 \\
 26 & 1.44035 & 56 & 1.58439 & 86 & 1.35253 & 116 & 1.36922 \\
 28 & 1.58439 & 58 & 1.36922 & 88 & 1.46703 & 118 & 1.34349 \\
 30 & 3.52086 & 60 & 3.52086 & 90 & 3.52086 & 120 & 3.52086 \\
 \hline
\end{array}
\end{equation*}
\caption{Numerical approximations of the Bateman--Horn constant $C_k$.}
\label{Table:CkList}
\end{table}

Although weaker than the Bateman--Horn conjecture, 
the Brun sieve \cite[Thm.~3, Sect.~I.4.2]{Tenenbaum} suffices for many applications.
It does, however, have the distinct advantage of being a proven fact, rather than a long-standing
conjecture. The Brun sieve implies that there
is a constant $B$ that depends only on $m$ and $D$ such that
\begin{equation*}
P(x)\leq  \frac{BC}{D}\int_2^x \frac{dt}{(\log t)^m}
=(1+o(1))\frac{BC}{D} \frac{x}{(\log x)^m} 
\end{equation*}
for sufficiently large $x$.  
In particular, there is a constant $K$ such that
\begin{equation*}
\pi_k(x)\,\leq\, K\frac{ C_kx}{(\log x)^2}
\end{equation*}
for all $k$ and sufficiently large $x$. 
Thus, the Brun sieve implies that the upper bound on $\pi_k$ implied
by the Bateman--Horn conjecture is of the correct order of magnitude.

\section{Totient comparison theorem}\label{Section:Comparison}

The well-known formula
\begin{equation}\label{eq:Phin}
	\frac{\phi(n)}{n} = \prod_{q |  n} \left( 1 - \frac{1}{q} \right).
\end{equation}
depends only on the primes that divide $n$ and not on their multiplicity.
Because of this, we find it more convenient to work with
\begin{equation*}
S(p) := \frac{\phi(p-1)}{p-1} - \frac{\phi(p+k-1)}{p+k-1}
\end{equation*}
instead of the more obvious quantity
\begin{equation*}
T(p)=\phi(p-1) - \phi(p+k-1).
\end{equation*}
We are able to do this because
the sign of $T(p)$ almost always agrees with the sign of $S(p)$.
For $k=2$, elementary considerations confirm that
$S(p)T(p) > 0$ for $p\geq 5$ \cite[Lem.~2]{GKL}.
For $k \geq 4$, the result is more difficult.  We require several
lemmas before we obtain an asymptotic analogue of the desired
result (Theorem \ref{Theorem:Comparison}).

We first need to estimate the number of $p\in \Pi_k(x)$ for which
$S(p)$ or $T(p)$ equals zero.  In both cases, the number is negligble
when compared with $\pi_k(x)$; this is Lemma \ref{Lemma:Equality} below.  
To this end, we need the following result.

\begin{lemma}[Graham--Holt--Pomerance \cite{Graham}]\label{Lemma:GHP}
Suppose that $j$ and $j+k$ have the same prime factors.
Let $g = \gcd(j,j+k)$ and suppose that
\begin{equation}\label{eq:qr}
 \frac{jt}{g} + 1 \qquad \text{and} \qquad \frac{(j+k)t}{g} + 1
\end{equation}
are primes that do not divide $j$.
\begin{enumerate}
\item Then $\displaystyle n = j \left( \frac{(j+k)t}{g} + 1 \right)$ satisfies $\phi(n) = \phi(n+k)$.
\item For $k$ fixed and sufficiently large $x$,
the number of solutions $n \leq x$ to $\phi(n) = \phi(n+k)$ that are not of the form above
is less than $x/\exp( (\log x)^{1/3})$.
\end{enumerate}
\end{lemma}

Part (b) of the preceding was improved by Yamada \cite{Yamada}, although the 
bound there is slightly more complicated than that of Graham--Holt--Pomerance.  In Lemma
\ref{Lemma:GHP}, one considers numbers with the same prime factors.  Because of this,
we will also need the following lemma of Thue.

\begin{lemma}[Thue \cite{Thue}]\label{Lemma:Thue}
Let $1 = n_1 < n_2 < \cdots$ be the sequence of positive integers whose prime factors
are at most $p$.  Then $\lim_{i\to\infty} (n_{i+1} - n_i) = \infty$.
\end{lemma}

A more explicit version of Thue's theorem is due to Tijdeman \cite{Tijdeman},
who proved that there is an effectively computable constant $C = C(p)$ such that
$n_{i+1} - n_i > n_i / (\log n_i)^C$ for $n_i \geq 3$.  For our purposes, however,
Thue's result is sufficient.  In particular, Lemma \ref{Lemma:Thue} implies that
for each fixed $k$, the sequence $n_1,n_2,\ldots$ contains only finitely many pairs $n_i,n_j$
for which $n_j = n_i + k$.

We are now ready to show that $T(p)$ and $S(p)$ are rarely equal to zero relative to the
counting function $\pi_k$.

\begin{lemma}\label{Lemma:Equality}
As $x \to \infty$,
\begin{enumerate}
\item $\#\{ p \in \Pi_k(x) : S(p) = 0 \} = o(\pi_k(x))$, and
\item $\#\{ p \in \Pi_k(x) : T(p) = 0 \} = o(\pi_k(x))$.
\end{enumerate}
\end{lemma}

\begin{proof}
\noindent(a)
Let $P(n)$ denote the largest prime factor of $n$. Since 
\begin{equation}\label{eq:qqpn}
\frac{\phi(n)}{n} = \prod_{q|n} \bigg(\frac{q-1}{q}\bigg),
\end{equation}
it follows that $P(n)$ is the largest prime factor of the denominator of $\phi(n)/n$.
If $S(p) = 0$, then $P(p-1) = P(p+k-1)$ divides
$\gcd(p-1,p+k-1)$, which divides $k$.  Consequently, $S(p) = 0$ implies that
the prime factors of both $p-1$ and $p+k-1$ are at most $k$.
Lemma \ref{Lemma:Thue} implies that only finitely many such $p$ exist.
Thus, the number of primes $p \in \Pi_k(x)$ for which $S(p) =0$ is $o(\pi_k(x))$.
\medskip

\noindent(b)
Lemma \ref{Lemma:Thue} ensures that for each fixed $k$, there 
are only finitely many $j$ for which $j$ and $j+k$ have the same prime factors.
Fix $j$ and let $g = \gcd(j,j+k)$.  To apply Lemma \ref{Lemma:GHP} with $n = p-1$,
we must count those
\begin{equation*}
\qquad\qquad t \leq \frac{g(x-j+1)}{j(j+k)} \qquad \text{(so that $p \leq x$)}
\end{equation*}
for which
\begin{align*}
p &= j\bigg( \frac{j+k}{g}t+1\bigg)+1,
&q &= j\bigg( \frac{j+k}{g}t+1\bigg)+k+1,\\
r&= \frac{j}{g}t = 1,
&s &= \frac{j+k}{g}t + 1,
\end{align*}
are simultaneously prime.  Since we have four linear constraints, 
the Brun sieve ensures that the number of such $t$ is
$O(x / (\log x)^4) = o(\pi_k(x))$.
Thus, the number of primes $p \in \Pi_k(x)$ for which $T(p) =0$ is $o(\pi_k(x))$.
\end{proof}

Our proof of Lemma \ref{Lemma:Equality}a actually shows something stronger:
 $S(p) = 0$ for only finitely many $p \in \Pi_k$.
We can prove Lemma \ref{Lemma:Equality}a as stated without Thue's result (Lemma \ref{Lemma:Thue})
as follows.
If $p-1 \leq x$ and $P(p-1) \leq k$, then
$p-1$ is divisible only by the $\pi(k)$ primes at most $k$.  The number of such 
$p$ at most $x$ is\footnote{If $\pi(k) = s$ and $P(p-1) \leq k$, we may write
$p-1 = p_1^{a_1} p_2^{a_2} \cdots p_s^{a_s}$, in which
$2 = p_1 < p_2 < \cdots < p_s$ are the primes at most $k$.
For $i=1,2,\ldots,s$, we have 
$p_i^{a_i} \leq x$ and hence $a_i \leq (\log x)/ \log p_i$.
Thus, there are at most $1+\log x/\log p_i$ possibilities for $a_i$.
Consequently, there are at most
$O((\log x+1)^s)=O((\log x)^{\pi(k)})$ admissible
 vectors of exponents $(a_1,a_2,\ldots,a_s)$.}
 $O((\log x)^{\pi(k)})=o(x/(\log x)^2)$,
even without the condition that $p$ is prime.

The next step toward the desired totient comparison theorem
(Theorem \ref{Theorem:Comparison}) is to prove that for each $\ell \geq 1$,
most $p \in \Pi_k(x)$ have the property that $2^{\ell} | T(p)$;
this is Lemma \ref{Lemma:2L}.   Since $T(p) = 0$ rarely occurs
by Lemma \ref{Lemma:Equality}b, it will follow that 
$T(p)$ is typically large in absolute value.  To do this,
we require the following folk lemma.  Since we are
unable to locate an exact reference for it, we provide the proof.

\begin{lemma}\label{Lemma:MertensPower}
$\displaystyle \sum_{q^a \leq x} \frac{1}{q^a} = \log \log x + O(1)$.
\end{lemma}

\begin{proof}
Mertens' theorem \cite[\S VII.28.1b]{Sandor} implies that
\begin{equation*}
\sum_{q \leq x} \frac{1}{q} = \log \log x + O(1).
\end{equation*}
Thus,
\begin{align*}
\sum_{q^a \leq x} \frac{1}{q^a}
&=  \sum_{q \leq x} \frac{1}{q} + \sum_{\substack{q^a \leq x \\ a \geq 2}} \frac{1}{q^a} 
\leq \log \log x + O(1) + \sum_{n \geq 2} \sum_{k\geq 2} \frac{1}{n^k} \\
&= \log \log x + O(1) + \sum_{n \geq 2} \frac{1}{n^2} \cdot \frac{1}{1-1/n}\\
&= \log \log x + O(1). \qedhere
\end{align*}
\end{proof}

Let $\omega(n)$ denote the number of distinct prime divisors of $n$.
The formula
\begin{equation*}
\phi(n) = \prod_{p^a \| n} p^{a-1}(p-1)
\end{equation*}
ensures that $2^{\omega(n)-1} | \phi(n)$ because each odd prime power $p^a$ that exactly divides $n$
provides at least an additional factor of $2$ to $\phi(n)$ since $p-1$ is even.  If $p$ is large,
then $p-1$ and $p+k-1$ tend to have many prime factors.  Thus, we expect that
$T(p)$ should be divisible by a large power of $2$.  The following makes this precise.

\begin{lemma}\label{Lemma:2L}
For $k \geq 2$ even and $\ell \geq 1$,
\begin{equation*}
\#\big\{ p \in \Pi_k(x) \,:\,2^{\ell}|T(p) \big\} \,\sim\,\pi_k(x).
\end{equation*}
\end{lemma}

\begin{proof}
It suffices to show that the counting function for the set of $p \in \Pi_k(x)$
for which $\omega(p-1) \leq \ell$ or $\omega(p+k-1) \leq \ell$ is $o(\pi_k(x))$.
Indeed, if $\omega(p-1),\omega(p+k-1) \geq \ell+1$, then the preceding discussion
implies that $2^{\ell}$ divides both $\phi(p-1)$ and $\phi(p+k-1)$, and hence divides
$T(p)$.

If $\omega(p-1) \leq \ell$, then
$p-1 = nr$, in which $r$ is prime and $\omega(n) \leq \ell$.
We must have $\gcd(n,k+1) =1$ since otherwise $p+k$ would be composite.
Let $g$ be the product of the three polynomials
\begin{equation*}
g_1(t) = t,\qquad
g_2(t) = nt+1,\qquad
g_3(t) = nt+k+1.
\end{equation*}
Then
\begin{equation*}
N_g(q) = \begin{cases}
	1 & \text{if $q | n$}, \\
	2 & \text{if $q \nmid n$ and $q|k$ or $q|(k+1)$},\\
	3 & \text{if $q \nmid n$, $q\nmid k$, $q\nmid (k+1)$}.
\end{cases}
\end{equation*}
The Brun sieve provides the following asymptotic estimate, uniformly in $n$:
\begin{align}
\sum_{\substack{ t \leq \frac{x}{n} \\ t, nt+1,\\ nt+k+1\, \text{prime}}} \!\!\!\!\!\!\!\!1
&\ll \frac{x/n}{(\log x/n)^3} \prod_q \frac{1 - N_g(q)/q}{(1 - 1/q)^3} \nonumber \\
&\ll \frac{x}{n(\log x)^3} \prod_{q|n} \frac{1 - 1/q}{(1 - 1/q)^3}
\prod_{\substack{q\nmid n \, \text{and}\\ \text{$q|k$ or $q|(k+1)$}}}  \frac{1 - 2/q}{(1 - 1/q)^3}
\prod_{\substack{q\nmid n , q\nmid k \\\text{and}\, q\nmid (k+1)}}  \frac{1 - 3/q}{(1 - 1/q)^3}  \nonumber \\
&\ll \frac{x}{n(\log x)^3} \bigg[\frac{1}{(1 - \frac{1}{2})^2}\bigg]^{\omega(n)}  \!\!\!\!
\prod_{\substack{q|k\,\text{or}\\q|(k+1)}}  \frac{1 - 2/q}{(1 - 1/q)^3}  \!\!\!\!
\prod_{\substack{q\nmid n , q\nmid k \\\text{and}\, q\nmid (k+1)}} \!\!\!\!\!\! \frac{1 - 3/q}{(1 - 1/q)^3}  \label{eq:TripleProduct}\\
&\ll \frac{2^{2\ell} x}{n(\log x)^3} . \nonumber
\end{align}
In the preceding computation, we used the fact that
\begin{equation*}
1 \leq \frac{1 - 2/q}{(1-1/q)^3}  \quad \text{for $q \geq 3$}
\end{equation*}
to overestimate the finite product in the middle of \eqref{eq:TripleProduct} independently of $n$.
Moreover, the third product in \eqref{eq:TripleProduct} converges since
\begin{equation*}
1 - \frac{1-3/q}{(1-1/q)^3} \,=\, \frac{3q-1}{(q-1)^3} \,\sim\, \frac{3}{q^2}.
\end{equation*}
Lemma \ref{Lemma:MertensPower} provides
\begin{align*}
\sum_{\substack{p \in \Pi_k(x)\\ \omega(p-1) \leq \ell}} 1
&= \sum_{\substack{n \leq x \\ \omega(n) \leq \ell}}  \sum_{\substack{ t \leq \frac{x}{n} \\ t, nt+1,\\ nt+k+1 \text{ prime}}} 1\\
&\ll  \sum_{\substack{n \leq x \\ \omega(n) \leq \ell}}  \frac{2^{2\ell} x}{n(\log x)^3} \\
&\ll  \frac{x}{(\log x)^3} \sum_{\substack{n \leq x \\ \omega(n) \leq \ell}}  \frac{1}{n} \\
&\ll\frac{x}{(\log x)^3} \frac{1}{\ell!}\bigg(1+\sum_{q^a \leq x} \frac{1}{q^a}\bigg)^{\ell} \\
&\ll\frac{x}{(\log x)^3} \big( \log \log x + O(1) \big)^{\ell} \\
&=o(\pi_k(x)).
\end{align*}
Similarly, the count of $p \in \Pi_k(x)$ with $\omega(p+k-1) \leq \ell$ is also $o(\pi_k(x))$.
\end{proof}

We are now in a position to prove the main result of this section.
It says that $S(p)$ and $T(p)$ are nonzero and 
share the same sign for most $p \in \Pi_k(x)$.

\begin{theorem}[Totient Comparison Theorem]\label{Theorem:Comparison}
Let $k$ be even.  Then as $x \to \infty$,
\begin{equation*}
\# \{ p \in \Pi_k(x) : S(p)T(p) > 0 \} \,\sim\, \pi_k(x).
\end{equation*}
\end{theorem}

\begin{proof}
In light of Lemma \ref{Lemma:Equality},
it suffices to show that
\begin{equation*}
\#\{ p \in \Pi_k(x) : S(p) > 0 ,\, T(p) > 0 \} \, \sim\, \pi_k(x).
\end{equation*}
Since $T(p) > 0$ implies that $S(p) > 0$, we focus on the 
converse.  If $S(p) > 0$, then
\begin{align}
0 
&< (p+k-1)\phi(p-1) - (p-1)\phi(p+k-1) \nonumber \\
&= p\big( \phi(p-1) - \phi(p+k-1)\big) + (k-1)\phi(p-1) + \phi(p+k-1) \nonumber \\
&\leq p\big( \phi(p-1) - \phi(p+k-1)\big) + (k-1)(p-1) + (p+k-1) \nonumber \\
&\leq p\big( \phi(p-1) - \phi(p+k-1) + k\big) \nonumber\\
&= p\big( T(p) + k\big) . \label{eq:Trouble}
\end{align}
Fix $\ell$ so that $2^{\ell} \geq k$.
Apply Lemma \ref{Lemma:2L} at \eqref{eq:Trouble} and conclude that
\begin{equation*}
\{ p \in \Pi_k(x) : S(p) > 0,\, T(p) \geq 0 \} \,\sim\, \pi_k(x).
\end{equation*}
Now apply Lemma \ref{Lemma:Equality}b to replace $T(p) \geq 0$ in the preceding
with $T(p) > 0$.
\end{proof}

In light of Theorem \ref{Theorem:Comparison}, we can focus our attention on 
the expression $S(p)$,
which is nonzero and shares the same sign as $T(p)$
for all $p \in \Pi_k$ outside of a set of zero density with respect to the counting function
$\pi_k(x)$.  The two expressions
\begin{equation}\label{eq:phicompare}
	\frac{\phi(p-1)}{p-1} = \prod_{q |  (p-1)} \left( 1 - \frac{1}{q} \right) \quad \text{and} \quad  \frac{\phi(p+k-1)}{p+k-1} = \prod_{q |  (p+k-1)} \left( 1 - \frac{1}{q} \right)
\end{equation}
that comprise $S(p)$ are primarily determined by the small prime divisors of $p-1$ and $p+k-1$. 
Since $p$ and $p+k$ are both prime, the nature of these small divisors is also related to $k$.

\section{An heuristic argument}\label{Section:Heuristic}

Before proceeding to the technical details, it is instructive to go through a brief heuristic argument.
With the help of the Bateman--Horn conjecture, we will ultimately be able to turn this
informal reasoning into rigorous, quantitative proofs.

As Table \ref{Table:PrimePairList} suggests, the behavior of $T(p)$ is heavily influenced by 
the value of $k \pmod{3}$.  Here is the explanation.
\begin{itemize}
\item If $k \equiv -1 \pmod{3}$, then elementary considerations imply that
$3 | (p+k-1)$ whenever $p,p+k$ are prime and $p \geq 5$.
Then \eqref{eq:phicompare} becomes
\begin{equation*}
\frac{\phi(p-1)}{p-1} 
= \frac{1}{2} \prod_{\substack{q \geq 5 \\ q |  (p-1) \\  q \nmid (k+1)}} \left( 1 - \frac{1}{q} \right) 
\quad\text{and}\quad  
\frac{\phi(p+k-1)}{p+k-1} 
= \frac{1}{3} \prod_{\substack{q \geq 5 \\ q |  (p+k-1) \\ q \nmid (k-1)}} \left( 1 - \frac{1}{q} \right),
\end{equation*}
and hence we expect that $S(p) > 0$ for most $p \in \Pi_k$.  Moreover, this suggests that
$S(p) < 0$ might occur if $p-1$ is divisible by many small primes.

\item If $k \equiv 1 \pmod{3}$, then a similar argument tells us that
\begin{equation*}
\frac{\phi(p-1)}{p-1} 
= \frac{1}{3} \prod_{\substack{q \geq 5 \\ q |  (p-1) \\ q \nmid (k+1)}} \left( 1 - \frac{1}{q} \right) 
\quad\text{and}\quad  \frac{\phi(p+k-1)}{p+k-1} = \frac{1}{2} \prod_{\substack{q \geq 5 \\ q |  (p+k-1) \\ q \nmid (k-1)}} \left( 1 - \frac{1}{q} \right).
\end{equation*}
Thus, we expect that $S(p) < 0$ for most $p \in \Pi_k$ and that $S(p) >0$ might occur if
$p+k-1$ is divisible by many small primes.

\item If $k \equiv 0 \pmod{q}$, in which $q$ is prime,
then $q$ either divides both $p-1$ and $p+k-1$, or it divides neither. 
Thus, the prime divisors of $k$ have no bearing upon the large-scale sign behavior of $S(p)$.
It is the small primes $q \geq 5$ that divide exactly one of $p-1$ and $p+k-1$ which govern
our problem.  Consequently, the observed bias in the sign of $S(p)$ is less pronounced
if $3|k$.  
\end{itemize}

%%%%%%%%%%%%%
\section{Primitive roots biases for $k \not\equiv 0 \pmod{3}$}\label{Section:k1}

Let $\chi_3$ denote the nontrivial Dirichlet character modulo $3$.  That is,
\begin{equation*}
	\chi_3(k) = \begin{cases}
		0 & \text{if $k \equiv 0 \pmod{3}$},\\
		1 & \text{if $k \equiv 1 \pmod{3}$},\\
		-1 & \text{if $k \equiv 2 \pmod{3}$}.
	\end{cases}
\end{equation*}
Fix $k \not \equiv 0 \pmod{3}$ and let
\begin{equation*}
5 \leq q_1 < q_2 < q_3 < \cdots
\end{equation*}
be the ordered sequence of primes that do not divide 
\begin{equation}\label{eq:zk}
k(k - \chi_3(k)),
\end{equation}
which is a multiple of $6$.
This sequence is infinite since it contains all primes larger than 
$\max\{5,k+1\}$.  
Let $Q = Q(k)$ denote the set
\begin{equation}\label{eq:Q}
	Q = \{q_1, q_2, \dots, q_m\},
\end{equation}
in which the index $m$ shall be determined momentarily. 
Define
\begin{equation}\label{eq:Lk}
	L_k = \log\left[ \frac{2}{3} \prod_{q \in Q} \left( 1 + \frac{1}{q-1} \right) \right]
\end{equation}
and
\begin{equation}\label{eq:Rk}
	R_k = \sum_{\substack{r \geq 5 \\ r \not\in Q \\ r \nmid (k+\chi_3(k))}} \frac{1}{r-N_f(r)} \log\left( 1 + \frac{1}{r-1} \right),
\end{equation}
in which $f(t) = t(t+k)$ is the polynomial defined in \eqref{eq:f}. 
From \eqref{eq:BHNfp}, we see that $$r - N_f(r) \in \{ r-1,r-2\}$$ for all primes $r$, so 
the general term in \eqref{eq:Rk} is $O(1/r^2)$.
Define $m$ in \eqref{eq:Q} to be the smallest index such that
\begin{equation}\label{eq:thmcond}
	L_k > R_k.
\end{equation}
This is possible since the product \eqref{eq:Lk} diverges if taken over all sufficiently large primes,
while the sum \eqref{eq:Rk} converges under the same circumstances.

This establishes the notation necessary for part (a) of the following result. For part (b), we use an
expression similar to \eqref{eq:Rk}.  Let
\begin{equation*}
	R_k' = \sum_{\substack{r \geq 5 \\ r \nmid (k - \chi_3(k))}} \frac{1}{r-N_f(r)} \log\left( 1 + \frac{1}{r-1} \right).
\end{equation*}
This lays the foundation for the following theorem, which establishes a bias 
in the number of primitive roots of prime pairs $p,p+k$ when $k \equiv \pm 1 \pmod{3}$.

\begin{theorem}\label{thm:kpm1mod3}
Assume that the Bateman--Horn conjecture holds.  Let $k \not\equiv 0 \pmod{3}$.
\begin{enumerate}
\item The set of primes $p \in \Pi_k$ for which
\begin{equation*}
	\sgn T(p) = \chi_3(k)
\end{equation*}
has lower density (as a subset of $\Pi_k$) at least
\begin{equation*}
	\prod_{q \in Q} (q-2)^{-1} \left(1 - \frac{R_k}{L_k} \right) > 0.
\end{equation*}

\item The set of primes $p \in \Pi_k$ for which
\begin{equation*}
	\sgn T(p) = -\chi_3(k)
\end{equation*}
has lower density (as a subset of all prime pairs $p, p+k$) at least
\begin{equation*}
	1 - \frac{R_k'}{\log(3/2)} > 0.6515.
\end{equation*}
\end{enumerate}
\end{theorem}

Tables \ref{table:k-1mod3} 
and \ref{table:k1mod3} 
provide the sets $Q$, numerical 
values for $L_k$, $R_k$, $R_k'$, and the bounds in Theorem \ref{thm:kpm1mod3} 
for various values of $k$.

\begin{table}\footnotesize
\centering
\begin{tabular}{|c|l|l|l|l|l|l|}
\hline
$k$ & \multicolumn{1}{c|}{$Q$} & \multicolumn{1}{c|}{$L_k$} & \multicolumn{1}{c|}{$R_k$} & \multicolumn{1}{c|}{\begin{tabular}[c]{@{}c@{}}Lower Bound:\\ $T(p) < 0$\end{tabular}} & \multicolumn{1}{c|}{$R_k'$} & \multicolumn{1}{c|}{\begin{tabular}[c]{@{}c@{}}Lower Bound:\\ $T(p) > 0$\end{tabular}} \\ \hline\hline
\multicolumn{1}{|c|}{2} & 5, 7, 11 & 0.067139 & 0.025497 & 0.004594 & 0.141298 & 0.651516 \\ \hline
\multicolumn{1}{|c|}{8} & 5, 7, 11 & 0.067139 & 0.025497 & 0.004594 & 0.141298 & 0.651516 \\ \hline
\multicolumn{1}{|c|}{14} & \begin{tabular}[c]{@{}l@{}}11, 13, 17, 19, \\ 23, 29, 31, 37, \\ 41, 43, 47, 53\end{tabular} & 0.113089 & 0.103683 & $1.56 \times 10^{-18}$ & 0.061779 & 0.847635 \\ \hline
\multicolumn{1}{|c|}{20} & \begin{tabular}[c]{@{}l@{}}11, 13, 17, 19, \\ 23, 29, 31, 37, \\ 41, 43, 47\end{tabular} & 0.094041 & 0.090599 & $3.50 \times 10^{-17}$ & 0.091873 & 0.773414 \\ \hline
\multicolumn{1}{|c|}{26} & 5, 7, 11 & 0.067139 & 0.024890 & 0.004661 & 0.140692 & 0.653012 \\ \hline
32 & 5, 7, 13 & 0.051872 & 0.027680 & 0.002826 & 0.130708 & 0.677634 \\ \hline
38 & 5, 7, 11 & 0.067139 & 0.0245373 & 0.004700 & 0.133845 & 0.669898 \\ \hline
44 & \begin{tabular}[c]{@{}l@{}}7, 13, 17, 19,\\ 23, 29, 31, 37,\\ 41\end{tabular} & 0.107845 & 0.088373 & $5.73 \times 10^{-13}$ & 0.065858 & 0.837574 \\ \hline
50 & \begin{tabular}[c]{@{}l@{}}7, 11, 13, 19,\\ 23, 29, 31\end{tabular} & 0.090439 & 0.066279 & $1.93 \times 10^{-9}$ & 0.118661 & 0.707346 \\ \hline
56 & 5, 11, 13, 17 & 0.053656 & 0.039870 & 0.000057 & 0.132979 & 0.672033 \\ \hline
62 & 5, 11, 13, 17 & 0.053656 & 0.044691 & 0.000037 & 0.11043 & 0.727645 \\ \hline
68 & 5, 7, 11 & 0.067140 & 0.025013 & 0.004647 & 0.138929 & 0.65736 \\ \hline
74 & \begin{tabular}[c]{@{}l@{}}7, 11, 13, 17,\\ 19, 23\end{tabular} & 0.083182 & 0.083047 & $6.14 \times 10^{-10}$ & 0.066895 & 0.835016 \\ \hline
80 & \begin{tabular}[c]{@{}l@{}}7, 11, 13, 17,\\ 19, 23\end{tabular} & 0.083182 & 0.064502 & $8.47 \times 10^{-8}$ & 0.122703 & 0.697378 \\ \hline
86 & 5, 7, 11 & 0.067139 & 0.0214415 & 0.005041 & 0.139985 & 0.654755 \\ \hline
92 & 5, 7, 11 & 0.067139 & 0.018124 & 0.005407 & 0.140071 & 0.654542 \\ \hline
98 & \begin{tabular}[c]{@{}l@{}}5, 13, 17, 19,\\ 23\end{tabular} & 0.056865 & 0.045054 & $1.17 \times 10^{-6}$ & 0.125570 & 0.690307 \\ \hline
104 & \begin{tabular}[c]{@{}l@{}}11, 17, 19, 23,\\ 29, 31, 37, 41,\\ 43, 47, 53, 59,\\ 61, 67, 71, 73,\\ 79\end{tabular} & 0.122425 & 0.114018 & $1.71 \times 10^{-28}$ & 0.035480 & 0.912495 \\ \hline
110 & \begin{tabular}[c]{@{}l@{}}7, 13, 17, 19,\\ 23, 29, 31, 41\end{tabular} & 0.080446 & 0.071049 & $1.29 \times 10^{-11}$ & 0.120861 & 0.701920 \\ \hline
116 & 5, 7, 11 & 0.067139 & 0.023334 & 0.004833 & 0.133975 & 0.669577 \\ \hline
122 & 5, 7, 11 & 0.067139 & 0.025492 & 0.004594 & 0.140660 & 0.653089 \\ \hline
128 & 5, 7, 11 & 0.067139 & 0.025434 & 0.004601 & 0.140724 & 0.652931 \\ \hline
134 & \begin{tabular}[c]{@{}l@{}}7, 11, 13, 17,\\ 19, 23, 29\end{tabular} & 0.118273 & 0.081959 & $4.29 \times 10^{-9}$ & 0.066913 & 0.834971 \\ \hline
140 & \begin{tabular}[c]{@{}l@{}}11, 13, 17, 19,\\ 23, 29, 31, 37,\\ 41, 43, 53\end{tabular} & 0.091583 & 0.085513 & $5.59 \times 10^{-17}$ & 0.117087 & 0.711229 \\ \hline
146 & 5, 11, 13, 17 & 0.053656 & 0.043706 & 0.000041 & 0.110465 & 0.727559 \\ \hline
152 & 5, 7, 11 & 0.067139 & 0.025276 & 0.004618 & 0.137080 & 0.661920 \\ \hline
158 & 5, 7, 11 & 0.067139 & 0.025453 & 0.004599 & 0.140922 & 0.652442 \\ \hline
164 & \begin{tabular}[c]{@{}l@{}}7, 13, 17, 19,\\ 23, 29, 31, 37,\\ 43\end{tabular} & 0.106682 & 0.090011 & $4.72 \times 10^{-13}$ & 0.056311 & 0.861120 \\ \hline
\end{tabular}
\caption{Lower bounds in Theorem \ref{thm:kpm1mod3} for $k \equiv -1 \pmod{3}$.}
\label{table:k-1mod3}
\end{table}

\begin{table}\footnotesize
\centering
\begin{tabular}{|c|l|l|l|l|l|l|}
\hline
$k$ & \multicolumn{1}{c|}{$Q$} & \multicolumn{1}{c|}{$L_k$} & \multicolumn{1}{c|}{$R_k$} & \multicolumn{1}{c|}{\begin{tabular}[c]{@{}c@{}}Lower Bound:\\ $T(p) > 0$\end{tabular}} & \multicolumn{1}{c|}{$R_k'$} & \multicolumn{1}{c|}{\begin{tabular}[c]{@{}c@{}}Lower Bound:\\ $T(p) < 0$\end{tabular}} \\ \hline\hline
\multicolumn{1}{|c|}{4} & 5,7,11 & 0.067139 & 0.025497 & 0.004594 & 0.141298 & 0.651516 \\ \hline
\multicolumn{1}{|c|}{10} & \begin{tabular}[c]{@{}l@{}}7, 11, 13, 17,\\ 19, 23\end{tabular} & 0.083182 & 0.064667 & $8.39 \times 10^{-8}$ & 0.122703 & 0.697378 \\ \hline
\multicolumn{1}{|c|}{16} & \begin{tabular}[c]{@{}l@{}}7, 11, 13, 17, \\ 19, 23, 29\end{tabular} & 0.118273 & 0.081963 & $4.28 \times 10^{-9}$ & 0.066917 & 0.834963 \\ \hline
\multicolumn{1}{|c|}{22} & \begin{tabular}[c]{@{}l@{}}5, 13, 17, 19,\\ 23\end{tabular} & 0.056865 & 0.049243 & $7.58 \times 10^{-7}$ & 0.109409 & 0.730164 \\ \hline
\multicolumn{1}{|c|}{28} & 5, 11, 13, 17 & 0.053656 & 0.038571 & 0.000063 & 0.13616 & 0.664189 \\ \hline
34 & 5, 7, 13 & 0.051872 & 0.028558 & 0.002723 & 0.130455 & 0.678257 \\ \hline
40 & \begin{tabular}[c]{@{}l@{}}7, 11, 17, 19,\\ 23, 29, 31\end{tabular} & 0.071021 & 0.068880 & $1.59 \times 10^{-10}$ & 0.115426 & 0.715324 \\ \hline
46 & \begin{tabular}[c]{@{}l@{}}7, 11, 13, 17,\\ 19, 29, 31\end{tabular} & 0.106611 & 0.082375 & $2.30 \times 10^{-9}$ & 0.066821 & 0.8352 \\ \hline
52 & 5, 7, 11 & 0.067139 & 0.024517 & 0.004702 & 0.13665 & 0.662979 \\ \hline
58 & 5, 7, 11 & 0.067139 & 0.025150 & 0.004632 & 0.138071 & 0.659474 \\ \hline
64 & 5, 11, 13, 17 & 0.053656 & 0.045009 & 0.000036 & 0.110468 & 0.727552 \\ \hline
70 & \begin{tabular}[c]{@{}l@{}}11, 13, 17, 19,\\ 29, 31, 37, 41,\\ 43, 47, 53, 59,\\ 61\end{tabular} & 0.102261 & 0.086419 & $1.81 \times 10^{-20}$ & 0.115448 & 0.715271 \\ \hline
76 & \begin{tabular}[c]{@{}l@{}}7, 11, 13, 17,\\ 23, 29, 31\end{tabular} & 0.096996 & 0.083859 & $1.11 \times 10^{-9}$ & 0.066740 & 0.835398 \\ \hline
82 & 5, 7, 11 & 0.067139 & 0.025331 & 0.004612 & 0.141282 & 0.651555 \\ \hline
88 & 5, 7, 13 & 0.051872 & 0.027621 & 0.002833 & 0.138939 & 0.657333 \\ \hline
94 & 5, 7, 11 & 0.067139 & 0.022306 & 0.004946 & 0.140157 & 0.65433 \\ \hline
100 & \begin{tabular}[c]{@{}l@{}}7, 13, 17, 19,\\ 23, 29, 31, 37\end{tabular} & 0.083152 & 0.071944 & $1.67 \times 10^{-11}$ & 0.112113 & 0.723496 \\ \hline
106 & \begin{tabular}[c]{@{}l@{}}11, 13, 17, 19,\\ 23, 29, 31, 37,\\ 41, 43, 47, 59\end{tabular} & 0.111135 & 0.108798 & $3.52 \times 10^{-19}$ & 0.036080 & 0.911017 \\ \hline
112 & 5, 11, 13, 17 & 0.053656 & 0.039790 & 0.000058 & 0.135377 & 0.666119 \\ \hline
118 & 5, 7, 11 & 0.067139 & 0.02145 & 0.005040 & 0.134016 & 0.669475 \\ \hline
124 & 5, 7, 11 & 0.067139 & 0.025459 & 0.004599 & 0.140627 & 0.65317 \\ \hline
130 & \begin{tabular}[c]{@{}l@{}}7, 11, 17, 19,\\ 23, 29, 31\end{tabular} & 0.071021 & 0.068848 & $1.62 \times 10^{-10}$ & 0.121523 & 0.700289 \\ \hline
136 & \begin{tabular}[c]{@{}l@{}}7, 11, 13, 19,\\ 23, 29, 31\end{tabular} & 0.090439 & 0.084567 & $4.69 \times 10^{-10}$ & 0.066664 & 0.835586 \\ \hline
142 & 5, 7, 11 & 0.067139 & 0.018217 & 0.005397 & 0.140817 & 0.652702 \\ \hline
148 & 5, 11, 13, 17 & 0.053656 & 0.044941 & 0.000036 & 0.110446 & 0.727606 \\ \hline
154 & \begin{tabular}[c]{@{}l@{}}5, 13, 19, 23,\\ 29, 31\end{tabular} & 0.064121 & 0.045715 & $3.11 \times 10^{-8}$ & 0.131059 & 0.676768 \\ \hline
160 & \begin{tabular}[c]{@{}l@{}}7, 11, 13, 17,\\ 19, 23\end{tabular} & 0.083182 & 0.064667 & $8.39 \times 10^{-8}$ & 0.122329 & 0.698299 \\ \hline
166 & \begin{tabular}[c]{@{}l@{}}7, 13, 17, 19,\\ 23, 29, 31, 37,\\ 41\end{tabular} & 0.107845 & 0.089968 & $5.26 \times 10^{-13}$ & 0.056325 & 0.861085 \\ \hline
172 & 5, 7, 11 & 0.067139 & 0.025449 & 0.004599 & 0.138104 & 0.659394 \\ \hline
178 & 5, 7, 11 & 0.067139 & 0.025464 & 0.004598 & 0.140997 & 0.652259 \\ \hline
184 & 5, 7, 11 & 0.067139 & 0.024618 & 0.004691 & 0.140922 & 0.652444 \\ \hline
\end{tabular}
\caption{Lower bounds in Theorem \ref{thm:kpm1mod3} for $k \equiv 1 \pmod{3}$.}
\label{table:k1mod3}
\end{table}

\subsection{Preliminary lemmas}
Before proceeding with the proof of Theorem \ref{thm:kpm1mod3}, 
we require a few preliminary results. 
Certain conditions in Lemmas \ref{lemma:qdivisors} and \ref{lemma:rdivisor}
are slightly more general than necessary.  
This is because they will later be applied when $k \equiv 0 \pmod{3}$ (Section \ref{Section:k0}).
For our present purposes (the proof of Theorem \ref{thm:kpm1mod3}), 
the set $Q$ in the following lemmas is as defined in the preceding section.

\begin{lemma}\label{lemma:qdivisors}
Assume that the Bateman--Horn conjecture holds. Let $k$ be a positive even integer and let
$Q$ be a finite set of primes such that 
\begin{equation*}
\quad
q\nmid k(k+1)
\qquad \text{(resp., $q\nmid k(k-1)$)},
\end{equation*}
for all $q \in Q$. The number of $p \in \Pi_k(x)$ such that 
\begin{equation*}
q |  (p-1)\qquad \text{(resp., $q |  (p+k-1)$)}
\end{equation*}
for all $q \in Q$ is
\begin{equation*}
	\pi_k'(x) \,=\, (1 + o(1))\, \pi_k(x) \prod_{q \in Q} (q-2)^{-1}.
\end{equation*}
\end{lemma}

\begin{proof}
Suppose that $q \nmid k (k+1)$ for all $q \in Q$, since the case 
$q \nmid k(k-1)$ is analogous. 
We wish to count the number of $p \in \Pi_k(x)$ such that $q |  (p-1)$ for all $q \in Q$. 
If $a = \prod_{q \in Q}q$, then the desired primes are those of the form
\begin{equation*}
	n = at + 1 \leq x \quad \text{such that} \quad n+k = at + k + 1 \text{ is prime.}
\end{equation*}
Let
\begin{equation*}
	g_1(t) = at + 1, \quad g_2(t) = at+k+1, \quad \text{and } g = g_1g_2.
\end{equation*}
In the Bateman--Horn conjecture with $s$ denoting an arbitrary prime, we have
\begin{equation}\label{eq:ngp}
	N_g(s) = \begin{cases}
		0 \quad \text{if } s \in Q, \\
		1 \quad \text{if } s |  k,\, s \not\in Q, \\
		2 \quad \text{if } s \nmid k,\ s \not\in Q.
	\end{cases}
\end{equation}
For sufficiently large $x$, 
the Bateman--Horn conjecture predicts that the number of such $t \leq (x-1)/a$ is
\begin{align}
	\pi_k'(x)
	&= (1+o(1)) \frac{(x-1)/a}{(\log((x-1)/a))^2} \prod_{s \geq 2} \frac{1-N_g(s)/s}{(1-1/s)^2} \nonumber \\
	& = (1+o(1)) \frac{x}{a(\log x)^2} \prod_{s \geq 2} \frac{1-N_g(s)/s}{(1-1/s)^2} \nonumber \\
	&= (1+o(1)) \frac{x}{a(\log x)^2} \prod_{s \in Q} \frac{1}{(1-1/s)^2} \prod_{s \not\in Q} \frac{1-N_g(s)/s}{(1-1/s)^2} \nonumber \\
	&= (1+o(1)) \frac{x}{a(\log x)^2} \prod_{s \in Q} (1-N_f(s)/s)^{-1} \prod_{s \geq 2} \frac{1-N_f(s)/s}{(1-1/s)^2} \nonumber,
\end{align}
in which $N_f(s)$ refers to \eqref{eq:BHNfp}. Simplifying further yields
\begin{align}
	\pi_k'(x)
	&= (1+o(1)) \frac{\pi_k(x)}{a} \prod_{s \in Q} (1-2/s)^{-1} \nonumber \\
	&\quad= (1+o(1)) \pi_k(x) \prod_{q \in Q} (q-2)^{-1}. \nonumber\qedhere
\end{align}
\end{proof}

\begin{lemma}\label{lemma:rdivisor}
Assume that the Bateman--Horn conjecture holds. Let $k$ be a positive even integer and let
$Q$ be a finite set of primes such that
\begin{equation*}
\quad
q\nmid k(k+1)
\qquad \text{(resp., $q\nmid k(k-1)$)},
\end{equation*}
for all $q \in Q$. Let $r \geq 5$ be a fixed prime not in $Q$ 
that satisfies
\begin{equation*}
r \nmid k(k-1)\qquad \text{(resp., $r \nmid k(k+1)$)}. 
\end{equation*}
The number of $p \in \Pi_k(x)$ such that 
\begin{equation*}
q |  (p-1)
\quad \text{and} \quad
r |  (p+k-1)\qquad \text{(resp., $r |  (p-1)$)},
\end{equation*}
for all $q \in Q$ is
\begin{equation*}
	\pi_{k,r}'(x) = (1+o(1))\, \frac{\pi_k(x)}{r - N_f(r)} \, \prod_{q \in Q} (q-2)^{-1},
\end{equation*}
in which $N_f(r)$ refers to \eqref{eq:BHNfp}.
\end{lemma}

\begin{proof}
Suppose that $q \nmid (k+1)$ for all $q \in Q$, since the case $q \nmid (k-1)$ is analogous.
Fix a prime $r \geq 5$ such that $r \nmid (k-1)$ and let $a = \prod_{q \in Q}q$.
The desired primes are precisely those of the form
\begin{equation*}
\text{$n = a j+1\leq x$  \quad such that \quad $n+k = aj + k + 1$ is prime
and $r |  (aj + k)$}.
\end{equation*}
In particular, $j$ must be of the form
\begin{equation*}
j = j_0 + r \ell, 
\end{equation*}
in which $j_0$ is the smallest positive integer such that 
$j_0\equiv -k a^{-1}  \pmod r$ (note that $a$ is invertible modulo $r$ since
$r \notin Q$). 
Let $b_{r}=aj_0+1$.  Then
\begin{equation}\label{eq:APr}
n=a r \ell+ b_r
\qquad \text{and}\qquad
n+k=a r\ell+(b_r + k)
\end{equation}
are both prime, $n\leq x$, and
\begin{equation*}
\ell\leq \frac{x-b_{r}}{a r}  .
\end{equation*}
In the Bateman--Horn conjecture, let 
\begin{equation*}
g_1(t)=art+b_{r},  \qquad g_2(t)=a r t+(b_{r}+k),
\quad\text{and}\quad g = g_1g_2.
\end{equation*}
With $s$ denoting an arbitrary prime,
$N_g(s)$ is as in \eqref{eq:ngp} except for $s=r$, in which case $N_g(r) = 0$.
Indeed, 
\begin{equation*}
g_1(t) \equiv b_r \equiv aj_0 + 1 \equiv -k+1 \not\equiv 0\pmod{r}
\quad\text{and}\quad
g_2(t) \equiv b_r + k \equiv 1 \pmod{r}
\end{equation*}
for all $t$.
As $x \to \infty$, the Bateman--Horn conjecture 
predicts that the number of such $\ell$ is
\begin{align}
\pi_{k,r}'(x)
&= (1+o(1)) \frac{(x-b_r)/(a r)}{( \log( (x - b_r)/(a r)))^2 } \prod_{s \geq 2} \left( \frac{1 - N_g(s)/s}{(1 - 1/s)^2} \right) \nonumber \\
&=  (1+o(1)) \frac{x}{ ar (\log x )^2 } \prod_{s \geq 2} \left( \frac{1 - N_g(s)/s}{(1 - 1/s)^2} \right) \nonumber \\
& =  (1+o(1)) \frac{x}{ar (\log x)^2} \prod_{s \in Q \text{ or } s=r} \left( \frac{1}{(1- 1/s)^2} \right) \prod_{s \not\in Q, s \neq r} \left( \frac{1 - N_g(s)/s}{ (1 -1/s)^2} \right) \nonumber\\
&=  (1+o(1)) \frac{x}{ar(\log x)^2} \prod_{s \in Q \text{ or } s=r} (1-N_f(s)/s)^{-1} \prod_{s\geq 2} \left( \frac{1 - N_f(s)/s}{ (1 -1/s)^2} \right)  \nonumber \\
& =  (1+o(1)) \frac{\pi_k(x)}{r-N_f(r)} \prod_{q \in Q} (q-2)^{-1}. \nonumber\qedhere
\end{align}
\end{proof}

\subsection{Proof of Theorem \ref{thm:kpm1mod3}a} 

In light of Theorem \ref{Theorem:Comparison}, we may use $S(p)$ and $T(p)$
interchangeably in what follows.
Suppose that $k \not \equiv 0 \pmod{3}$.  
\begin{itemize}
\item If $\chi_3(k) = -1$, then we wish to count $p \in \Pi_k$ for which $q |  (p-1)$ for all $q \in Q$. 
\item If $\chi_3(k) = 1$, then we wish to count $p \in \Pi_k$ for which $q |  (p+k-1)$ for all $q \in Q$.
\end{itemize}
Because of this slight difference, we define
$\tau_k = k(1 +\chi_3(k)) / 2$.  That is,
\begin{equation}\label{eq:tauk}
	\tau_k = \begin{cases}
		0 & \text{if $\chi_3(k) = -1$}, \\
		k & \text{if $\chi_3(k) = 1$},
	\end{cases}
\end{equation}
so that
\begin{equation*}
	p-1+\tau_k = \begin{cases}
		p-1 & \text{if $\chi_3(k) = -1$}, \\
		p+k-1 & \text{if $\chi_3(k) = 1$}.
	\end{cases}
\end{equation*}
Now let $\pi_k'(x)$ denote the number of primes $p \in \Pi_k(x)$ such that 
$q |  (p-1+\tau_k)$ for all $q \in Q$. 
Lemma \ref{lemma:qdivisors} allows us to count these prime pairs. 
Moreover, Lemma \ref{lemma:rdivisor} permits us to counts such pairs after imposing the 
additional restriction that a 
fixed prime $r \geq 5$ not in $Q$ 
divides $p-1+(k-\tau_k)$, where
\begin{equation*}
	p-1+(k-\tau_k) = \begin{cases}
		p+k-1 & \text{if $\chi_3(k) = -1$}, \\
		p-1 & \text{if $\chi_3(k) = 1$}.
	\end{cases}
\end{equation*}
Let $\pi_{k,r}'(x)$ denote the number of primes $p \in \Pi_k(x)$ such that 
$q |  (p-1+\tau_k)$ for all $q \in Q$, and $r |  (p-1+(k-\tau_k))$.

Suppose that $p$ is counted by $\pi_k'(x)$. The condition $k \not\equiv 0 \pmod{q}$ ensures that 
$q \nmid (p-1+(k-\tau_k))$ for all $q \in Q$.  Thus,
\begin{equation*}
3 |  (p-1+(k-\tau_k)) \qquad \text{and} \qquad 3 \nmid (p-1+\tau_k),
\end{equation*}
so that
\begin{equation*}
	\frac{\phi(p-1+\tau_k)}{p-1+\tau_k} \leq \frac{1}{2} \prod_{q \in Q} \left( 1 - \frac{1}{q} \right).
\end{equation*}
If $\sgn S(p) = -\chi_3(k)$ (so that $p$ does not belong to the set of interest in
Theorem \ref{thm:kpm1mod3}a), then
\begin{align*}
	\frac{1}{3} \prod_{\substack{r |  (p-1+(k-\tau_k)) \\ r \geq 5, r \not\in Q}} \left(1 - \frac{1}{r}\right)
	&\,=\, \frac{\phi(p-1+(k-\tau_k))}{p-1+(k-\tau_k)} \\
	&\,<\, \frac{\phi(p-1+\tau_k)}{p-1+\tau_k} 
	\,\leq\, \frac{1}{2} \prod_{q \in Q} \left( 1 - \frac{1}{q} \right).
\end{align*}
Consequently,
\begin{equation*}
	\prod_{\substack{r |  (p-1+(k-\tau_k)) \\ r \geq 5, r \notin Q}} \left(1+\frac{1}{r-1}\right)
	\,>\, \frac{2}{3} \prod_{q \in Q} \left(1+\frac{1}{q-1}\right),
\end{equation*}
in which $r$ is prime. Let
\begin{equation*}
	F(p) := \sum_{\substack{r |  (p-1+(k-\tau_k)) \\ r \geq 5}} \log\left(1 + \frac{1}{r-1}\right).
\end{equation*}

We want to count primes $p \in \Pi_k(x)$ such that
\begin{equation*}
	F(p) > \log\left[ \frac{2}{3} \prod_{q \in Q} \left(1+\frac{1}{q-1}\right) \right] = L_k
\end{equation*}
and $q |  (p-1+\tau_k)$ for all $q \in Q$. To do this, we first sum up $F(p)$ over all primes $p$ counted by $\pi_{k}'(x)$ and change the order of summation to get
\begin{align}
A(x)
&=\sum_{\substack{\text{$p$ counted by}\\ \pi_{k}'(x)}} F(p) \nonumber\\
&=\sum_{\substack{r \geq 5 \\ r \not\in Q}} \pi_{k,r}'(x)  \log\left(1+\frac{1}{r-1}\right) \nonumber\\
&\leq  \sum_{\substack{5\leq r\leq z \\ r \not\in Q}} \pi_{k,r}'(x) \log \left(1+\frac{1}{r-1}\right)\nonumber\\
&\qquad\qquad  +  \sum_{\substack{z<r\leq (\log x)^3 \\ r \not\in Q}} \pi_{k,r}'(x)\log\left(1+\frac{1}{r-1}\right)\nonumber\\
 & \qquad\qquad\qquad +  \sum_{\substack{(\log x)^3<r\leq x \\ r \not\in Q}} \pi_{k,r}'(x)\log\left(1+\frac{1}{r-1}\right) \nonumber\\
 & = A_1(x)+A_2(x)+A_3(x), \label{eq:AAA}
\end{align}
in which $z$ is a fixed number.  We bound the three summands separately. In what follows, we let 
$\delta > 0$ be small, and fix $z$ large enough so that 
\begin{equation*}
	\frac{8K}{z-2} \prod_{q \in Q} (q-2)^{-1} < \frac{\delta}{3}.	
\end{equation*}

\begin{enumerate} 
\item Suppose that $5\leq r\leq z$ and $r \not\in Q$. 
Lemma \ref{lemma:rdivisor} asserts that if $r \nmid (k+\chi_3(k))$, then
\begin{equation*}
 \pi_{k,r}'(x) = (1+o(1)) \frac{\pi_k(x)}{r-N_f(r)} \prod_{q \in Q} (q-2)^{-1}
\end{equation*}
uniformly for $r \in [5,z] \setminus Q$ as $x \to \infty$. If $r |  (k+\chi_3(k))$ and $r |  (p-1+(k-\tau_k))$, then
\begin{equation*}
	0 \equiv p - 1 + (k-\tau_k) \equiv p - 1 - (\chi_3(k) + \tau_k) \pmod{r}.
\end{equation*}
When $\chi_3(k) = -1$, we have $p \equiv 0 \pmod{r}$. When $\chi_3(k) = 1$, we add $k+\chi_3(k)$ to the middle expression and simplify to get $p+k \equiv 0 \pmod{3}$. In either case, it follows that 
$\pi_{k,r}'(x) \leq 1$ when $r |  (k+\chi_3(k))$. Thus, for sufficiently large $x$ we have
\begin{align*}
A_1(x) 
& \leq  (1+o(1)) \frac{\pi_k(x)}{\prod_{q \in Q} (q-2)} \sum_{\substack{5 \leq r \leq z \\ r \not\in Q, r \nmid (k+\chi_3(k))}} \frac{1}{r-N_f(r)}\log\left( 1 + \frac{1}{r-1} \right) \\
&\qquad + \sum_{\substack{5 \leq r \leq z \\ r \not\in Q, r |  (k+\chi_3(k))}} \log\left( 1 + \frac{1}{r-1} \right) \\
& \leq  (1+o(1)) \frac{R_k \, \pi_k(x)}{\prod_{q \in Q} (q-2)} + \sum_{\substack{5 \leq r \leq z \\ r \not\in Q, r |  (k+\chi_3(k))}} \log\left( 1 + \frac{1}{r-1} \right) \\
& =  \left( (1+o(1))\, \frac{R_k}{\prod_{q \in Q} (q-2)} + \frac{1}{\pi_k(x)} \!\!\!\!\!\!\!\sum_{\substack{5 \leq r \leq z \\ r \not\in Q, r |  (k+\chi_3(k))}} \log\left( 1 + \frac{1}{r-1} \right) \right) \, \pi_k(x) \\
&< \left( \frac{R_k}{\prod_{q \in Q} (q-2)} + \frac{\delta}{3} \right)\, \pi_k(x),
\end{align*}
where the last inequality follows from $\pi_k(x) \to \infty$.

\item Suppose that $z<r\leq (\log x)^3$ and $r \not\in Q$. 
Maintaining the notation $a, b_r$ from the proof of Lemma \ref{lemma:rdivisor}, 
the Brun sieve yields an absolute constant $K$ such that for sufficiently large $x$,
\begin{align*}
	\pi_{k,r}'(x)
	&\leq \frac{K(x+\tau_k-b_r)/ar}{(\log((x+\tau_k-b_r)/ar))^2} \prod_{p \geq 2} \frac{1-N_g(p)/p}{(1-1/p)^2} \\
	&= \frac{C_k K(x+\tau_k-b_r)}{(r-N_f(r)) (\log((x+\tau_k-b_r)/ar))^2} \prod_{q \in Q}(q-2)^{-1} \\
	&\leq \frac{2 C_k K x}{(r-N_f(r)) (\log((x-b_r)/ar))^2} \prod_{q \in Q}(q-2)^{-1} \\
	&\leq \frac{2 C_k Kx}{(r-N_f(r)) (\log(x/ar - 1))^2} \prod_{q \in Q}(q-2)^{-1},
\end{align*}
where the last inequality follows from the fact that $b_r \leq ar$. Since $r \leq (\log x)^3$,
\begin{equation*}
	\log(x/ar-1) \geq \log(x^{1/2}) = (\log x)/2
\end{equation*}
for large enough $x$. Thus,
\begin{align*}
	\pi_{k,r}'(x)
	&\leq \frac{8C_kKx}{(r-N_f(r)) (\log x)^2} \prod_{q \in Q}(q-2)^{-1} \\
	&\leq \frac{8K\pi_k(x)}{r-2} \prod_{q \in Q} (q-2)^{-1}
\end{align*}
for sufficiently large $x$. 
Since $\log(1+t)<t$ for $t>0$, for sufficiently large $x$ we obtain
\begin{align*}
A_2(x) 
&= \sum_{\substack{z<r\leq (\log x)^3 \\ r \not\in Q}} \pi_{k,r}'(x)\log\left(1+\frac{1}{r-1}\right) \\
& \leq  \frac{8K \pi_k(x)}{\prod_{q \in Q} (q-2)} \sum_{r>z} \frac{1}{r-2} \log\left(1+\frac{1}{r-1}\right) \\
& <   \frac{8K \pi_k(x)}{\prod_{q \in Q} (q-2)} \sum_{r>z} \frac{1}{(r-2)(r-1)}\\
&  = \frac{8K \pi_k(x)}{\prod_{q \in Q} (q-2)}  \sum_{r>z} \left(\frac{1}{r-2}-\frac{1}{r-1}\right)\\
& \leq  \frac{8K \pi_k(x)}{(z-2) \prod_{q \in Q} (q-2)} \\
&< \frac{\delta}{3} \,\pi_k(x).
\end{align*}

\item Suppose that $(\log x)^3<r\leq x$ and $r \not\in Q$. 
By \eqref{eq:APr}, the primes counted by $\pi_{k,r}'(x)$ lie in 
an arithmetic progression modulo $ar$, with $a$ defined as in Lemma \ref{lemma:rdivisor}.  Thus, their number is at most 
\begin{equation*}
\pi_{k,r}'(x)\leq \left\lfloor \frac{x}{a r}\right\rfloor+1\leq \frac{x}{a r}+1.
\end{equation*}
Since $\log(1+t)<t$, for sufficiently large $x$ we obtain
\begin{align*}
A_3(x) 
&= \sum_{\substack{(\log x)^3<r\leq x \\ r \not\in Q}} \pi_{k,r}'(x)\log\left(1+\frac{1}{r-1}\right) \\
& \leq   \sum_{\substack{(\log x)^3<r\leq x \\ r \not\in Q}} \frac{1}{(r-1)} \left(\frac{x}{a r}+1\right)\\
& \leq  \frac{x}{a} \sum_{r>(\log x)^3} \frac{1}{r(r-1)}+\sum_{(\log x)^3<r\leq x} \frac{1}{r-1}\\
& \leq  \frac{x}{a} \sum_{r > (\log x)^3} \left(\frac{1}{r-1}-\frac{1}{r}\right)+\int_{(\log x)^3-2}^{x} \frac{dt}{t}\\
& \leq  \frac{x}{a((\log x)^3-1)}+\left(\log t\Big|_{t=(\log x)^3-2}^{t=x}\right)\\
& \leq  \frac{x}{a ((\log x)^3-1)}+\log x \\
& = (1+o(1))\left( \frac{x}{a(\log x)^3} + \log x \right) \\
%& =  (1+o(1) \left( \frac{1}{a C_k \log x} +   \frac{(\log x)^3}{C_kx} \right)\frac{C_k x}{(\log x)^2}\\
%&= (1+o(1))\left(  \frac{1}{aC_k \log x} + \frac{ (\log x)^3}{C_k x} \right) \pi_k(x) \\
&= o(1)\, \pi_k(x) \\
&< \frac{\delta}{3}\, \pi_k(x).
\end{align*}
\end{enumerate}

Returning to \eqref{eq:AAA} and using the preceding three estimates, we have
\begin{align*}
A(x) 
& =  A_1(x)+A_2(x)+A_3(x)\\
& <   \left( \frac{ R_k}{\prod_{q \in Q} (q-2)} + \delta \right) \pi_k(x)
\end{align*}
for sufficiently large $x$. Let $\mathcal{U}(x)$ be the set of primes $p$ counted by $\pi_k'(x)$ with $\sgn S(p) = -\chi_3(k)$, so that $p$ does not belong to the set of interest in 
Theorem \ref{thm:kpm1mod3}a. As we have seen,
if $p\in \mathcal{U}(x)$, then
\begin{equation*}
	F(p) > L_k.
\end{equation*}
Thus,
\begin{align*}
	0 
	&\leq \#\mathcal{U}(x) \, L_k \\
	&< \sum_{p\in \mathcal{U}(x)} F(p)\leq A(x) \\
	&< \left( \frac{ R_k}{\prod_{q \in Q} (q-2)} + \delta \right) \pi_k(x),
\end{align*}
from which we deduce that
\begin{equation*}
	\# \mathcal{U}(x)
	< \left(\frac{R_k \prod_{q \in Q} (q-2)^{-1} + \delta}{L_k} \right)\, \pi_k(x).
\end{equation*}
The primes $p$ counted by $\pi_{k}'(x)$ which are not in $\mathcal{U}(x)$ satisfy $\sgn S(p) = \chi_3(k)$. By Lemma \ref{lemma:qdivisors} and the preceding calculation,
for large $x$ there are at least
\begin{align*}
\pi_k'(x)-\#\mathcal{U}(x) 
& >  \left( (1+o(1))\, \prod_{q \in Q} (q-2)^{-1} - \frac{R_k \prod_{q \in Q} (q-2)^{-1} + \delta}{L_k} \right) \pi_k(x) \\
&= \prod_{q \in Q} (q-2)^{-1} \left( 1 - \frac{R_k}{L_k} - \epsilon \right) \pi_k(x)
\end{align*}
such primes, where $\epsilon > 0$ can be made arbitrarily small by taking $x$ large enough. The condition \eqref{eq:thmcond} ensures that the quantity in parentheses is positive for a small enough $\epsilon$. 
By Theorem \ref{Theorem:Comparison}, the set of $p \in \Pi_k$ for which
\begin{equation*}
	\sgn S(p) = \sgn T(p)
\end{equation*}
has full density as a subset of $\Pi_k$. 
It follows that the set of prime pairs for which $\sgn T(p) = \chi_3(k)$ has lower density
\begin{align*}
	\liminf_{x \to \infty} \frac{ \{ p \in \Pi_k(x) :  \sgn T(p)=\chi_3(k) \} }{\pi_k(x)} \geq \prod_{q \in Q} (q-2)^{-1} \left( 1 - \frac{R_k}{S_k} - \epsilon \right)  + o(1).
\end{align*}
Because this holds for all $\epsilon > 0$, the lower bound in Theorem \ref{thm:kpm1mod3}a follows.

\subsection{Proof of Theorem \ref{thm:kpm1mod3}b}
As before, we may use $S(p)$ and $T(p)$
interchangeably in what follows.
Fix $k$ satisfying $\chi_3(k) = \pm 1$ and let $r \geq 5$ be prime. 
We wish to count the number of $p \in \Pi_k(x)$ for which $r | ( p-1+\tau_k )$. 

If $r |  (k - \chi_3(k))$, then 
\begin{equation*}
	p - 1 + \tau_k \pm (k - \chi_3(k)) \equiv 0 \pmod{r}.
\end{equation*}
Consequently, \eqref{eq:tauk} permits us to deduce that
$p+k \equiv 0 \pmod{r}$ of $p \equiv 0 \pmod{r}$.
In either case, 
there is at most one such prime $p$. 

Now suppose that $r \nmid (k-\chi_3(k))$ and let
\begin{equation*}
	g_1(t) = rt + 1 - \tau_k, \quad g_2(t) = rt + 1 + (k-\tau_k), \quad \text{and} \quad g = g_1g_2.
\end{equation*}
Then
\begin{equation*}
	N_g(p) = \begin{cases}
		0 & \text{if } p = r, \\
		1 & \text{if } p |  k, \\
		2 & \text{if } p \nmid k,
	\end{cases}
\end{equation*}
so the Bateman--Horn conjecture gives
\begin{align}
	\sum_{\substack{p \in \Pi_k(x) \\ p+\tau_k \equiv 1 \pmod{r} \\ r \nmid (k-\chi_3(k))}} 1
	\;&=\; (1+o(1)) \frac{(x+\tau_k-1)/r}{(\log ((x+\tau_k-1)/r))^2} \prod_{p \geq 2} \frac{1-N_g(p)/p}{(1-1/p)^2} \nonumber \\
	&= (1+o(1)) \frac{x}{r(\log x)^2} \cdot \frac{1}{(1-1/r)^2} \prod_{p \ne r} \frac{1-N_g(p)/p}{(1-1/p)^2} \nonumber \\
	&= (1+o(1)) \frac{x}{(\log x)^2} \cdot \frac{1}{r - N_f(r)} \prod_{p \geq 2} \frac{1-N_f(p)/p}{(1-1/p)^2} \nonumber \\
	&= (1+o(1))\frac{\pi_k(x)}{r-N_f(r)}, \label{eq:p1modr2}
\end{align}
in which $N_f(r)$ refers to \eqref{eq:BHNfp}. If $\sgn S(p) = \chi_3(k)$, so that $p$ does 
not belong to the set of interest in Theorem \ref{thm:kpm1mod3}b, then
\begin{equation*}
	\frac{1}{2} \prod_{\substack{r |  (p-1+\tau_k) \\ r \geq 5}} \left(1 - \frac{1}{r}\right)
	\,=\, \frac{\phi(p-1+\tau_k)}{p-1+\tau_k}
	\,<\, \frac{\phi(p-1+(k-\tau_k))}{p-1+(k-\tau_k)}
	\,\leq\, \frac{1}{3},
\end{equation*}
because $3 \nmid (p-1+\tau_k)$ and $3 |  (p-1+(k-\tau_k))$. If
\begin{equation*}
	G(p)
	:= \sum_{\substack{r | (p-1+\tau_k) \\ r \geq 5}} \log\left( 1 + \frac{1}{r-1} \right),
\end{equation*}
then $G(p) > \log(3/2)$ whenever $p, p+k$ are primes that satisfy
$\sgn S(p) = \chi_3(k)$. Let $\pi_k''(x)$ denote the number of primes $p \in \Pi_k(x)$ 
for which $\sgn S(p) > \chi_3(k)$. 
For sufficiently large $x$, \eqref{eq:p1modr2} implies that
\begin{align*}
	\pi_k''(x) \log(3/2)
	&< \sum_{p \in \Pi_k(x)} G(p) \\
	&= \sum_{p \in \Pi_k(x)} \; \sum_{\substack{r |  (p-1+\tau_k) \\ r \geq 5}} \log\left( 1 + \frac{1}{r-1} \right) \\
	&\leq \sum_{5 \leq r \leq x} \log\left( 1 + \frac{1}{r-1} \right) \;  \sum_{\substack{p \in \Pi_k(x) \\ p+\tau_k \equiv 1 \pmod{r}}} 1 \\
	&\leq \sum_{\substack{5 \leq r \leq x \\ r \nmid (k-\chi_3(k))}} \log\left( 1 + \frac{1}{r-1} \right) \;  \sum_{\substack{p \in \Pi_k(x) \\ p+k \equiv 1 \pmod{r}}} 1 \\
	&\qquad + \sum_{\substack{5 \leq r \leq x \\ r |  (k-\chi_3(k))}} \log\left( 1 + \frac{1}{r-1} \right) \\
	&\leq \bigg[ (1+o(1)) \sum_{\substack{r \geq 5 \\ r \nmid (k-\chi_3(k))}} \frac{1}{r-N_f(r)}\log\left(1 + \frac{1}{r-1}\right) \\
	&\qquad + \frac{1}{\pi_k(x)}\sum_{\substack{r \geq 5 \\ r |  (k-\chi_3(k))}} \log\left( 1 + \frac{1}{r-1} \right) \bigg] \, \pi_k(x) \\
	&= \left( \sum_{\substack{r \geq 5 \\ r \nmid(k-\chi_3(k))}} \frac{1}{r-N_f(r)}\log\left(1 + \frac{1}{r-1}\right) + o(1) \right) \, \pi_k(x) \\
	&= (R_k' + o(1))\, \pi_k(x).
\end{align*}
Thus, there are at least
\begin{align*}
	\pi_k(x) - \pi_k''(x) > \pi_k(x)\left( 1 - \frac{R_k'}{\log(3/2)} - o(1) \right)
\end{align*}
primes $p \in \Pi_k(x)$ such $\sgn S(p) = -\chi_3(k)$. 
Reasoning similar to that used in 
the conclusion of the proof of part (a) yield the formula in Theorem \ref{thm:kpm1mod3}b.

To show that this lower density is bounded below by $0.6515$, we observe that\footnote{The terms of $R_k'$ are $O(1/r^2)$, since $t < \log(1+t)$ for $t>0$, so the series converges.
\texttt{Mathematica} provides the numerical value $0.141298112$.}
\begin{align*}
	R_k'
	&\,=\, \sum_{\substack{r \geq 5 \\ r \nmid(k-\chi_3(k))}} \frac{1}{r-N_f(r)}\log\left(1 + \frac{1}{r-1}\right) \\
	&\,\leq\, \sum_{r \geq 5} \frac{1}{r-2}\log\left(1 + \frac{1}{r-1}\right) \\
	&\,<\, 0.1412981.
\end{align*}
It follows that
\begin{align*}
	1 - \frac{R_k'}{\log(3/2)}
	\,>\, 1 - \frac{0.1412981}{\log(3/2)}
	\,>\,  0.6515.
\end{align*}
This completes the proof of Theorem \ref{thm:kpm1mod3}.\qed

\section{Extending Theorem \ref{thm:kpm1mod3} to Pairs $p, p+k$ with $k \equiv 0 \pmod{3}$}
\label{Section:k0}
Fix $k \equiv 0 \pmod{3}$.  
The techniques used in the proof of Theorem \ref{thm:kpm1mod3}a can be used to show that 
$T(p) < 0$ and $T(p) > 0$ both occur with positive density as a subset of $\Pi_k$.
Because the proofs are nearly identical, we simply point out the small differences and
leave the remaining details to the reader. 

Since $\chi_3(k) = 0$ whenever $k \equiv 0 \pmod{3}$, some notational adjustment is needed. 
To show that $T(p) < 0$ occurs with positive density in $\Pi_k$, 
we follow the proof of Theorem \ref{thm:kpm1mod3}a as if $k \equiv -1 \pmod{3}$, 
replacing each occurrence of $\chi_3(k)$ with $-1$.
Similarly, to show that $T(p) > 0$ occurs with positive density, we follow the proof as if
$k \equiv 1 \pmod{3}$, replacing $\chi_3(k)$ with $1$.

We modify the definition of $L_k$ by setting
\begin{equation*}
	L_k^{\pm} := \log\left[ \prod_{q \in Q^{\pm}} \left(1 + \frac{1}{q-1} \right) \right] ,
\end{equation*}
in which $Q^{\pm}$ are finite sets of primes to be determined shortly.
Note the absence of the $2/3$ factor inside the logarithm. 
This is due to the fact that $3$ either divides both $p-1$ and $p+k-1$, or it divides neither.
Consequently, the usual $2/3$ from \eqref{eq:Lk} is ``canceled" when we compare 
$\phi(p-1)/(p-1)$ and $\phi(p+k-1)/(p+k-1)$.  This is also the reason why
we cannot employ the techniques from the proof Theorem \ref{thm:kpm1mod3}b to establish a lower density greater than $0.5$ when $k \equiv 0 \pmod{3}$.  This is not surprising, since
Table \ref{Table:PrimePairList} demonstrates that
there is no universal bias in the sign of $T(p)$ that applies for all $k \equiv 0 \pmod{3}$.

Next, we let 
\begin{equation*}
	R_k^{\pm} = \sum_{\substack{r \geq 5 \\ r \not\in Q^{\pm} \\ r \nmid (k\pm 1)}} \frac{1}{r-N_f(r)} \log\left( 1 + \frac{1}{r-1} \right),
\end{equation*}
in which the signs are chosen depending on whether we wish to prove
$T(p) > 0$ or $T(p) < 0$.  We define $Q^{\pm}$ to be the smallest ordered subset of primes for which
$q \nmid k(k \mp 1)$ for all $q \in Q^{\pm}$ and such that
\begin{equation*}
	L_k^{\pm} > R_k^{\pm}.
\end{equation*}
Beyond the aforementioned, the only other difference in the proof is the absence of the $2/3$ factor when comparing $\phi(p-1)/(p-1)$ and $\phi(p+k-1)/(p+k-1)$. With this in mind, we have the following result.

\begin{theorem}\label{thm:k0mod3}
Assume that the Bateman--Horn conjecture holds. If $k \equiv 0 \pmod{3}$, then
the set of primes $p \in \Pi_k$ for which
\begin{equation*}
	\sgn T(p) = \pm 1
\end{equation*}
has lower density (as a subset of $\Pi_k$) at least
\begin{equation*}
	\prod_{q \in Q^{\pm}} (q-2)^{-1} \left(1 - \frac{R_k^{\pm}}{L_k^{\pm}} \right) > 0.
\end{equation*}
\end{theorem}

Table \ref{table:k0mod3} provides 
numerical values for $R_k^{\pm}$, $L_k^{\pm}$, 
and the bounds in Theorem \ref{thm:k0mod3} for various values of $k \equiv 0 \pmod{3}$.

\begin{table}\footnotesize
\centering
\begin{tabular}{c|c|l|l|l||c|l|l|l|l|}
\cline{2-9}
 & \multicolumn{4}{c||}{} & \multicolumn{4}{c|}{} \\[-1em]
 & \multicolumn{4}{c||}{$\chi_3(k)=-1$} & \multicolumn{4}{c|}{$\chi_3(k)=1$} \\ \hline
\multicolumn{1}{|c|}{$k$} & \multicolumn{1}{c|}{$Q^-$} & \multicolumn{1}{c|}{$L_k^-$} & \multicolumn{1}{c|}{$R_k^-$} & \multicolumn{1}{c||}{\begin{tabular}[c]{@{}c@{}}Lwr Bd:\\ $T(p)<0$\end{tabular}} & \multicolumn{1}{c|}{$Q^+$} & \multicolumn{1}{c|}{$L_k^+$} & \multicolumn{1}{c|}{$R_k^+$} & \multicolumn{1}{c|}{\begin{tabular}[c]{@{}c@{}}Lwr Bd:\\ $T(p)>0$\end{tabular}} \\ \hline\hline
\multicolumn{1}{|c|}{6} & 5 & 0.223144 & 0.066917 & 0.233372 & 7 & 0.154151 & 0.110468 & 0.056675 \\ \hline
\multicolumn{1}{|c|}{12} & 5 & 0.223144 & 0.056327 & 0.249192 & 5 & 0.223144 & 0.059640 & 0.244242 \\ \hline
\multicolumn{1}{|c|}{18} & 5 & 0.223144 & 0.062875 & 0.23941 & 5 & 0.223144 & 0.063737 & 0.238123 \\ \hline
\multicolumn{1}{|c|}{24} & 7 & 0.154151 & 0.108351 & 0.059422 & 5 & 0.223144 & 0.066917 & 0.233372 \\ \hline
\multicolumn{1}{|c|}{30} & 7 & 0.154151 & 0.090573 & 0.082487 & 7 & 0.154151 & 0.090742 & 0.082268 \\ \hline
\multicolumn{1}{|c|}{36} & 5 & 0.223144 & 0.036087 & 0.279427 & 11, 13 & 0.175353 & 0.122649 & 0.003035 \\ \hline
\multicolumn{1}{|c|}{42} & 5 & 0.223144 & 0.061145 & 0.241994 & 5 & 0.223144 & 0.061205 & 0.241905 \\ \hline
\multicolumn{1}{|c|}{48} & 5 & 0.223144 & 0.066439 & 0.234086 & 5 & 0.223144 & 0.036087 & 0.279427 \\ \hline
\multicolumn{1}{|c|}{54} & 7 & 0.154151 & 0.110094 & 0.057159 & 5 & 0.223144 & 0.056327 & 0.249192 \\ \hline
\multicolumn{1}{|c|}{60} & 7 & 0.154151 & 0.091573 & 0.081190 & 7 & 0.154151 & 0.091593 & 0.081164 \\ \hline
\multicolumn{1}{|c|}{66} & 5 & 0.223144 & 0.058581 & 0.245824 & 7 & 0.154151 & 0.109178 & 0.058349 \\ \hline
\multicolumn{1}{|c|}{72} & 5 & 0.223144 & 0.066711 & 0.233679 & 5 & 0.223144 & 0.066723 & 0.233663 \\ \hline
\multicolumn{1}{|c|}{78} & 5 & 0.223144 & 0.024890 & 0.296152 & 5 & 0.223144 & 0.066145 & 0.234525 \\ \hline
\multicolumn{1}{|c|}{84} & 11, 13 & 0.175353 & 0.118143 & 0.003295 & 5 & 0.223144 & 0.057737 & 0.247086 \\ \hline
\multicolumn{1}{|c|}{90} & 11, 17 & 0.155935 & 0.107941 & 0.002279 & 7 & 0.154151 & 0.084596 & 0.090242 \\ \hline
\multicolumn{1}{|c|}{96} & 5 & 0.223144 & 0.063737 & 0.238123 & 7 & 0.154151 & 0.110359 & 0.056816 \\ \hline
\multicolumn{1}{|c|}{102} & 5 & 0.223144 & 0.066564 & 0.2339 & 5 & 0.223144 & 0.066568 & 0.233894 \\ \hline
\multicolumn{1}{|c|}{108} & 5 & 0.223144 & 0.066828 & 0.233506 & 5 & 0.223144 & 0.066831 & 0.233501 \\ \hline
\multicolumn{1}{|c|}{114} & 7 & 0.154151 & 0.110211 & 0.057008 & 5 & 0.223144 & 0.064624 & 0.236798 \\ \hline
\multicolumn{1}{|l|}{120} & 7 & 0.154151 & 0.087831 & 0.086045 & 11, 13 & 0.175353 & 0.104836 & 0.004062 \\ \hline
\multicolumn{1}{|l|}{126} & 5 & 0.223144 & 0.061779 & 0.241048 & 11, 13 & 0.175353 & 0.11823 & 0.003290 \\ \hline
\multicolumn{1}{|l|}{132} & 5 & 0.223144 & 0.065799 & 0.235043 & 5 & 0.223144 & 0.031847 & 0.28576 \\ \hline
\multicolumn{1}{|l|}{138} & 5 & 0.223144 & 0.066766 & 0.233597 & 5 & 0.223144 & 0.066768 & 0.233595 \\ \hline
\multicolumn{1}{|l|}{144} & 7 & 0.154151 & 0.092601 & 0.079856 & 5 & 0.223144 & 0.065617 & 0.235314 \\ \hline
\multicolumn{1}{|l|}{150} & 7 & 0.154151 & 0.091827 & 0.080860 & 7 & 0.154151 & 0.091828 & 0.080859 \\ \hline
\multicolumn{1}{|l|}{156} & 5 & 0.223144 & 0.065180 & 0.235967 & 7 & 0.154151 & 0.10982 & 0.0575156 \\ \hline
\end{tabular}
\caption{Lower bounds from Theorem \ref{thm:k0mod3} for $k \equiv 0 \pmod{3}$.}
\label{table:k0mod3}
\end{table}

\bibliographystyle{plain}

\bibliography{PRBPP1ENTB}

\end{document}